\documentclass[12pt,fleqn]{article}

\usepackage{amsmath,amssymb,amsthm}
\usepackage{geometry}
\usepackage[pdfpagemode=UseNone,pdfstartview=FitH,hypertex]{hyperref}

\newcommand{\abs}[1]{\left|#1\right|}
\newcommand{\bdry}[1]{\partial #1}
\newcommand{\calN}{{\cal N}}
\newcommand{\closure}[1]{\overline{#1}}
\newcommand{\dint}{\ds{\int}}
\newcommand{\ds}[1]{\displaystyle #1}
\newcommand{\eps}{\varepsilon}
\newcommand{\norm}[2][]{\left\|#2\right\|_{#1}}
\renewcommand{\o}{\text{o}}
\newcommand{\PS}[1]{$(\text{PS})_{#1}$}
\newcommand{\R}{\mathbb R}
\newcommand{\seq}[1]{\left(#1\right)}
\newcommand{\set}[1]{\left\{#1\right\}}
\newcommand{\vol}[1]{\left|#1\right|}
\newcommand{\wstar}{\xrightarrow{w^\ast}}

\newenvironment{enumroman}{\begin{enumerate}

}{\end{enumerate}}

\newtheorem{lemma}{Lemma}[section]
\newtheorem{theorem}[lemma]{Theorem}
\newtheorem{remark}[lemma]{Remark}

\numberwithin{equation}{section}

\title{\vspace{-0.5in} \bf A class of semipositone $p$-Laplacian problems with a critical growth reaction term}
\author{\bf Kanishka Perera\\
Department of Mathematical Sciences\\
Florida Institute of Technology\\
Melbourne, FL 32901, USA\\
\em kperera@fit.edu\\
[\bigskipamount]
\bf Ratnasingham Shivaji\\
Department of Mathematics and Statistics\\
University of North Carolina at Greensboro\\
Greensboro, NC 27412, USA\\
\em shivaji@uncg.edu\\
[\bigskipamount]
\bf Inbo Sim\\
Department of Mathematics\\
University of Ulsan\\
Ulsan 680-749, Republic of Korea\\
\em ibsim@ulsan.ac.kr}
\date{}

\begin{document}

\maketitle

\begin{abstract}
We prove the existence of ground state positive solutions for a class of semipositone $p$-Laplacian problems with a critical growth reaction term. The proofs are established by obtaining crucial uniform $C^{1,\alpha}$ a priori estimates and by concentration compactness arguments. Our results are new even in the semilinear case $p = 2$.\\[3pt]
\end{abstract}

\section{Introduction}

Consider the $p$-superlinear semipositone $p$-Laplacian problem
\begin{equation} \label{1}
\left\{\begin{aligned}
- \Delta_p\, u & = u^{q-1} - \mu && \text{in } \Omega\\[5pt]
u & > 0 && \text{in } \Omega\\[5pt]
u & = 0 && \text{on } \bdry{\Omega},
\end{aligned}\right.
\end{equation}
where $\Omega$ is a smooth bounded domain in $\R^N$, $1 < p < N$, $p < q \le p^\ast$, $\mu > 0$ is a parameter, and $p^\ast = Np/(N - p)$ is the critical Sobolev exponent. The scaling $u \mapsto \mu^{1/(q-1)}\, u$ transforms the first equation in \eqref{1} into
\[
- \Delta_p\, u = \mu^{(q-p)/(q-1)} \left(u^{q-1} - 1\right),
\]
so in the subcritical case $q < p^\ast$, it follows from the results in Castro et al.\! \cite{MR3507282} and Chhetri et al.\! \cite{MR3406455} that this problem has a weak positive solution for sufficiently small $\mu > 0$ when $p > 1$ (see also Unsurangie \cite{Unsurangie}, Allegretto et al.\! \cite{MR1141729}, Ambrosetti et al.\! \cite{MR1270096}, and Caldwell et al.\! \cite{MR2328697} for the case when $p = 2$). On the other hand, in the critical case $q = p^\ast$, it follows from a standard argument involving the Pohozaev identity for the $p$-Laplacian (see Guedda and V{\'e}ron \cite[Theorem 1.1]{MR1009077}) that problem \eqref{1} has no solution for any $\mu > 0$ when $\Omega$ is star-shaped. The purpose of the present paper is to show that this situation can be reversed by the addition of lower-order terms, as was observed in the positone case by Br{\'e}zis and Nirenberg in the celebrated paper \cite{MR709644}. However, this extension to the semipositone case is not straightforward as $u = 0$ is no longer a subsolution, making it much harder to find a positive solution as was pointed out in Lions \cite{MR678562}. The positive solutions that we obtain here are ground states, i.e., they minimize the energy among all positive solutions.

We study the Br{\'e}zis-Nirenberg type critical semipositone $p$-Laplacian problem
\begin{equation} \label{2}
\left\{\begin{aligned}
- \Delta_p\, u & = \lambda u^{p-1} + u^{p^\ast - 1} - \mu && \text{in } \Omega\\[5pt]
u & > 0 && \text{in } \Omega\\[5pt]
u & = 0 && \text{on } \bdry{\Omega},
\end{aligned}\right.
\end{equation}
where $\lambda, \mu > 0$ are parameters. Let $W^{1,p}_0(\Omega)$ be the usual Sobolev space with the norm given by
\[
\norm{u}^p = \int_\Omega |\nabla u|^p\, dx.
\]
For a given $\lambda > 0$, the energy of a weak solution $u \in W^{1,p}_0(\Omega)$ of problem \eqref{2} is given by
\[
I_\mu(u) = \int_\Omega \bigg(\frac{|\nabla u|^p}{p} - \frac{\lambda u^p}{p} - \frac{u^{p^\ast}}{p^\ast} + \mu u\bigg)\, dx,
\]
and clearly all weak solutions lie on the set
\[
\calN_\mu = \set{u \in W^{1,p}_0(\Omega) : u > 0 \text{ in } \Omega \text{ and } \int_\Omega |\nabla u|^p\, dx = \int_\Omega \left(\lambda u^p + u^{p^\ast} - \mu u\right) dx}.
\]
We will refer to a weak solution that minimizes $I_\mu$ on $\calN_\mu$ as a ground state. Let
\begin{equation} \label{3}
\lambda_1 = \inf_{u \in W^{1,p}_0(\Omega) \setminus \set{0}}\, \frac{\dint_\Omega |\nabla u|^p\, dx}{\dint_\Omega |u|^p\, dx}
\end{equation}
be the first Dirichlet eigenvalue of the $p$-Laplacian, which is positive. We will prove the following existence theorem.

\begin{theorem} \label{Theorem 1}
If $N \ge p^2$ and $\lambda \in (0,\lambda_1)$, then there exists $\mu^\ast > 0$ such that for all $\mu \in (0,\mu^\ast)$, problem \eqref{2} has a ground state solution $u_\mu \in C^{1,\alpha}(\closure{\Omega})$ for some $\alpha \in (0,1)$.
\end{theorem}

The scaling $u \mapsto \mu^{-1/(p^\ast - p)}\, u$ transforms the first equation in the critical semipositone $p$-Laplacian problem
\begin{equation} \label{31}
\left\{\begin{aligned}
- \Delta_p\, u & = \lambda u^{p-1} + \mu \left(u^{p^\ast - 1} - 1\right) && \text{in } \Omega\\[5pt]
u & > 0 && \text{in } \Omega\\[5pt]
u & = 0 && \text{on } \bdry{\Omega}
\end{aligned}\right.
\end{equation}
into
\[
- \Delta_p\, u = \lambda u^{p-1} + u^{p^\ast - 1} - \mu^{(p^\ast - 1)/(p^\ast - p)},
\]
so as an immediate corollary we have the following existence theorem for problem \eqref{31}.

\begin{theorem} \label{Theorem 2}
If $N \ge p^2$ and $\lambda \in (0,\lambda_1)$, then there exists $\mu^\ast > 0$ such that for all $\mu \in (0,\mu^\ast)$, problem \eqref{31} has a ground state solution $u_\mu \in C^{1,\alpha}(\closure{\Omega})$ for some $\alpha \in (0,1)$.
\end{theorem}

We would like to emphasize that Theorems \ref{Theorem 1} and \ref{Theorem 2} are new even in the semilinear case $p = 2$.

The outline of the proof of Theorem \ref{Theorem 1} is as follows. We consider the modified problem
\begin{equation} \label{4}
\left\{\begin{aligned}
- \Delta_p\, u & = \lambda u_+^{p-1} + u_+^{p^\ast - 1} - \mu\, f(u) && \text{in } \Omega\\[5pt]
u & = 0 && \text{on } \bdry{\Omega},
\end{aligned}\right.
\end{equation}
where $u_+(x) = \max \set{u(x),0}$ and
\[
f(t) = \begin{cases}
1, & t \ge 0\\[5pt]
1 - |t|^{p-1}, & -1 < t < 0\\[5pt]
0, & t \le -1.
\end{cases}
\]
Weak solutions of this problem coincide with critical points of the $C^1$-functional
\begin{multline*}
I_\mu(u) = \int_\Omega \bigg(\frac{|\nabla u|^p}{p} - \frac{\lambda u_+^p}{p} - \frac{u_+^{p^\ast}}{p^\ast}\bigg)\, dx + \mu\, \bigg[\int_{\set{u \ge 0}} u\, dx\\[10pt]
+ \int_{\set{-1 < u < 0}} \left(u - \frac{|u|^{p-1}\, u}{p}\right) dx - \left(1 - \frac{1}{p}\right) \vol{\set{u \le -1}}\bigg], \quad u \in W^{1,p}_0(\Omega),
\end{multline*}
where $\vol{\cdot}$ denotes the Lebesgue measure in $\R^N$. Recall that $I_\mu$ satisfies the Palais-Smale compactness condition at the level $c \in \R$, or the \PS{c} condition for short, if every sequence $\seq{u_j} \subset W^{1,p}_0(\Omega)$ such that $I_\mu(u_j) \to c$ and $I_\mu'(u_j) \to 0$, called a \PS{c} sequence for $I_\mu$, has a convergent subsequence. As we will see in Lemma \ref{Lemma 1} in the next section, it follows from concentration compactness arguments that $I_\mu$ satisfies the \PS{c} condition for all
\[
c < \frac{1}{N}\, S^{N/p} - \left(1 - \frac{1}{p}\right) \mu \vol{\Omega},
\]
where $S$ is the best Sobolev constant (see \eqref{5}). First we will construct a mountain pass level below this threshold for compactness for all sufficiently small $\mu > 0$. This part of the proof is more or less standard. The novelty of the paper lies in the fact that the solution $u_\mu$ of the modified problem \eqref{4} thus obtained is positive, and hence also a solution of our original problem \eqref{2}, if $\mu$ is further restricted. Note that this does not follow from the strong maximum principle as usual since $- \mu\, f(0) < 0$. This is precisely the main difficulty in finding positive solutions of semipositone problems (see Lions \cite{MR678562}). We will prove that for every sequence $\mu_j \to 0$, a subsequence of $u_{\mu_j}$ is positive in $\Omega$. The idea is to show that a subsequence of $u_{\mu_j}$ converges in $C^1_0(\closure{\Omega})$ to a solution of the limit problem
\[
\left\{\begin{aligned}
- \Delta_p\, u & = \lambda u^{p-1} + u^{p^\ast - 1} && \text{in } \Omega\\[5pt]
u & > 0 && \text{in } \Omega\\[5pt]
u & = 0 && \text{on } \bdry{\Omega}.
\end{aligned}\right.
\]
This requires a uniform $C^{1,\alpha}(\closure{\Omega})$ estimate of $u_{\mu_j}$ for some $\alpha \in (0,1)$. We will obtain such an estimate by showing that $u_{\mu_j}$ is uniformly bounded in $W^{1,p}_0(\Omega)$ and uniformly equi-integrable in $L^{p^\ast}(\Omega)$, and applying a result of de Figueiredo et al.\! \cite{MR2530603}. The proof of uniform equi-integrability in $L^{p^\ast}(\Omega)$ involves a second (nonstandard) application of the concentration compactness principle. Finally, we use the mountain pass characterization of our solution to show that it is indeed a ground state.

\begin{remark}
Establishing the existence of solutions to the critical semipositone problem
\[
\left\{\begin{aligned}
- \Delta_p\, u & = \mu \left(u^{p-1} + u^{p^\ast-1}-1\right) && \text{in } \Omega\\[5pt]
u & > 0 && \text{in } \Omega\\[5pt]
u & = 0 && \text{on } \bdry{\Omega}
\end{aligned}\right.
\]
for small $\mu$ remains open.
\end{remark}

\section{Preliminaries}

Let
\begin{equation} \label{5}
S = \inf_{u \in W^{1,p}_0(\Omega) \setminus \set{0}}\, \frac{\dint_\Omega |\nabla u|^p\, dx}{\left(\dint_\Omega |u|^{p^\ast}\, dx\right)^{p/p^\ast}}
\end{equation}
be the best constant in the Sobolev inequality, which is independent of $\Omega$. The proof of Theorem \ref{Theorem 1} will make use of the following compactness result.

\begin{lemma} \label{Lemma 1}
For any fixed $\lambda, \mu > 0$, $I_\mu$ satisfies the {\em \PS{c}} condition for all
\begin{equation} \label{6}
c < \frac{1}{N}\, S^{N/p} - \left(1 - \frac{1}{p}\right) \mu \vol{\Omega}.
\end{equation}
\end{lemma}

\begin{proof}
Let $\seq{u_j}$ be a \PS{c} sequence. First we show that $\seq{u_j}$ is bounded. We have
\begin{multline} \label{7}
I_\mu(u_j) = \int_\Omega \bigg(\frac{|\nabla u_j|^p}{p} - \frac{\lambda u_{j+}^p}{p} - \frac{u_{j+}^{p^\ast}}{p^\ast}\bigg)\, dx + \mu\, \bigg[\int_{\set{u_j \ge 0}} u_j\, dx\\[10pt]
+ \int_{\set{-1 < u_j < 0}} \left(u_j - \frac{|u_j|^{p-1}\, u_j}{p}\right) dx - \left(1 - \frac{1}{p}\right) \vol{\set{u_j \le -1}}\bigg] = c + \o(1)
\end{multline}
and
\begin{multline} \label{8}
I_\mu'(u_j)\, v = \int_\Omega \left(|\nabla u_j|^{p-2}\, \nabla u_j \cdot \nabla v - \lambda u_{j+}^{p-1}\, v - u_{j+}^{p^\ast - 1}\, v\right) dx + \mu\, \bigg[\int_{\set{u_j \ge 0}} v\, dx\\[10pt]
+ \int_{\set{-1 < u_j < 0}} \left(1 - |u_j|^{p-1}\right) v\, dx\bigg] = \o(1) \norm{v} \quad \forall v \in W^{1,p}_0(\Omega).
\end{multline}
Taking $v = u_j$ in \eqref{8}, dividing by $p$, and subtracting from \eqref{7} gives
\begin{equation} \label{9}
\frac{1}{N} \int_\Omega u_{j+}^{p^\ast}\, dx \le c + \left(1 - \frac{1}{p}\right) \mu \vol{\Omega} + \o(1) \left(\norm{u_j} + 1\right),
\end{equation}
and it follows from this, \eqref{7}, and the H\"{o}lder inequality that $\seq{u_j}$ is bounded in $W^{1,p}_0(\Omega)$.

Since $\seq{u_j}$ is bounded, so is $\seq{u_{j+}}$, a renamed subsequence of which then converges to some $v \ge 0$ weakly in $W^{1,p}_0(\Omega)$, strongly in $L^q(\Omega)$ for all $q \in [1,p^\ast)$ and a.e.\! in $\Omega$, and
\begin{equation} \label{10}
|\nabla u_{j+}|^p\, dx \wstar \kappa, \qquad u_{j+}^{p^\ast}\, dx \wstar \nu
\end{equation}
in the sense of measures, where $\kappa$ and $\nu$ are bounded nonnegative measures on $\closure{\Omega}$ (see, e.g., Folland \cite{MR1681462}). By the concentration compactness principle of Lions \cite{MR834360,MR850686}, then there exist an at most countable index set $I$ and points $x_i \in \closure{\Omega},\, i \in I$ such that
\begin{equation} \label{11}
\kappa \ge |\nabla v|^p\, dx + \sum_{i \in I} \kappa_i\, \delta_{x_i}, \qquad \nu = v^{p^\ast}\, dx + \sum_{i \in I} \nu_i\, \delta_{x_i},
\end{equation}
where $\kappa_i, \nu_i > 0$ and $\nu_i^{p/p^\ast} \le \kappa_i/S$. We claim that $I = \emptyset$. Suppose by contradiction that there exists $i \in I$. Let $\varphi : \R^N \to [0,1]$ be a smooth function such that $\varphi(x) = 1$ for $|x| \le 1$ and $\varphi(x) = 0$ for $|x| \ge 2$. Then set
\[
\varphi_{i,\rho}(x) = \varphi\left(\frac{x - x_i}{\rho}\right), \quad x \in \R^N
\]
for $i \in I$ and $\rho > 0$, and note that $\varphi_{i,\rho} : \R^N \to [0,1]$ is a smooth function such that $\varphi_{i,\rho}(x) = 1$ for $|x - x_i| \le \rho$ and $\varphi_{i,\rho}(x) = 0$ for $|x - x_i| \ge 2 \rho$. The sequence $\seq{\varphi_{i,\rho}\, u_{j+}}$ is bounded in $W^{1,p}_0(\Omega)$ and hence taking $v = \varphi_{i,\rho}\, u_{j+}$ in \eqref{8} gives
\begin{multline} \label{12}
\int_\Omega \big(\varphi_{i,\rho}\, |\nabla u_{j+}|^p + u_{j+}\, |\nabla u_{j+}|^{p-2}\, \nabla u_{j+} \cdot \nabla \varphi_{i,\rho} - \lambda\, \varphi_{i,\rho}\, u_{j+}^p - \varphi_{i,\rho}\, u_{j+}^{p^\ast}\\[5pt]
+ \mu\, \varphi_{i,\rho}\, u_{j+}\big)\, dx = \o(1).
\end{multline}
By \eqref{10},
\[
\int_\Omega \varphi_{i,\rho}\, |\nabla u_{j+}|^p\, dx \to \int_\Omega \varphi_{i,\rho}\, d\kappa, \qquad \int_\Omega \varphi_{i,\rho}\, u_{j+}^{p^\ast}\, dx \to \int_\Omega \varphi_{i,\rho}\, d\nu.
\]
Denoting by $C$ a generic positive constant independent of $j$ and $\rho$,
\begin{multline*}
\abs{\int_\Omega \big(u_{j+}\, |\nabla u_{j+}|^{p-2}\, \nabla u_{j+} \cdot \nabla \varphi_{i,\rho} - \lambda\, \varphi_{i,\rho}\, u_{j+}^p + \mu\, \varphi_{i,\rho}\, u_{j+}\big)\, dx}\\[10pt]
\le C \left[\left(\frac{1}{\rho} + \mu\right) I_j^{1/p} + I_j\right],
\end{multline*}
where
\[
I_j := \int_{\Omega \cap B_{2 \rho}(x_i)} u_{j+}^p\, dx \to \int_{\Omega \cap B_{2 \rho}(x_i)} v^p\, dx \le C \rho^p \left(\int_{\Omega \cap B_{2 \rho}(x_i)} v^{p^\ast}\, dx\right)^{p/p^\ast}.
\]
So passing to the limit in \eqref{12} gives
\[
\int_\Omega \varphi_{i,\rho}\, d\kappa - \int_\Omega \varphi_{i,\rho}\, d\nu \le C \left[(1 + \mu \rho) \left(\int_{\Omega \cap B_{2 \rho}(x_i)} v^{p^\ast}\, dx\right)^{1/p^\ast} + \int_{\Omega \cap B_{2 \rho}(x_i)} v^p\, dx\right].
\]
Letting $\rho \searrow 0$ and using \eqref{11} now gives $\kappa_i \le \nu_i$, which together with $\nu_i > 0$ and $\nu_i^{p/p^\ast} \le \kappa_i/S$ then gives $\nu_i \ge S^{N/p}$. On the other hand, passing to the limit in \eqref{9} and using \eqref{10} and \eqref{11} gives
\[
\nu_i \le N \left[c + \left(1 - \frac{1}{p}\right) \mu \vol{\Omega}\right] < S^{N/p}
\]
by \eqref{6}, a contradiction. Hence $I = \emptyset$ and
\begin{equation} \label{13}
\int_\Omega u_{j+}^{p^\ast}\, dx \to \int_\Omega v^{p^\ast}\, dx.
\end{equation}

Passing to a further subsequence, $u_j$ converges to some $u$ weakly in $W^{1,p}_0(\Omega)$, strongly in $L^q(\Omega)$ for all $q \in [1,p^\ast)$, and a.e.\! in $\Omega$. Since
\[
|u_{j+}^{p^\ast - 1}\, (u_j - u)| \le u_{j+}^{p^\ast} + u_{j+}^{p^\ast - 1}\, |u| \le \left(2 - \frac{1}{p^\ast}\right) u_{j+}^{p^\ast} + \frac{1}{p^\ast}\, |u|^{p^\ast}
\]
by Young's inequality,
\[
\int_\Omega u_{j+}^{p^\ast - 1}\, (u_j - u)\, dx \to 0
\]
by \eqref{13} and the dominated convergence theorem. Then taking $v = u_j - u$ in \eqref{8} gives
\[
\int_\Omega |\nabla u_j|^{p-2}\, \nabla u_j \cdot \nabla (u_j - u)\, dx \to 0,
\]
so $u_j \to u$ in $W^{1,p}_0(\Omega)$ for a renamed subsequence (see, e.g., Perera et al.\! \cite[Proposition 1.3]{MR2640827}).
\end{proof}

The infimum in \eqref{5} is attained by the family of functions
\[
u_\eps(x) = \frac{C_{N,p}\, \eps^{(N-p)/p^2}}{(\eps + |x|^{p/(p-1)})^{(N-p)/p}}, \quad \eps > 0
\]
when $\Omega = \R^N$, where the constant $C_{N,p} > 0$ is chosen so that
\[
\int_{\R^N} |\nabla u_\eps|^p\, dx = \int_{\R^N} u_\eps^{p^\ast}\, dx = S^{N/p}.
\]
Without loss of generality, we may assume that $0 \in \Omega$. Let $r > 0$ be so small that $B_{2r}(0) \subset \Omega$, take a function $\psi \in C^\infty_0(B_{2r}(0),[0,1])$ such that $\psi = 1$ on $B_r(0)$, and set
\[
\tilde{u}_\eps(x) = \psi(x)\, u_\eps(x), \qquad v_\eps(x) = \frac{\tilde{u}_\eps(x)}{\left(\dint_\Omega \tilde{u}_\eps^{p^\ast}\, dx\right)^{1/p^\ast}},
\]
so that $\dint_\Omega v_\eps^{p^\ast}\, dx = 1$. Then we have the well-known estimates
\begin{gather}
\label{14} \int_\Omega |\nabla v_\eps|^p\, dx \le S + C \eps^{(N-p)/p},\\[5pt]
\label{15} \int_\Omega v_\eps^p\, dx \ge \begin{cases}
\dfrac{1}{C}\, \eps^{p-1}, & N > p^2\\[10pt]
\dfrac{1}{C}\, \eps^{p-1}\, |\!\log \eps|, & N = p^2,
\end{cases}
\end{gather}
where $C = C(N,p) > 0$ is a constant (see, e.g., Dr{\'a}bek and Huang \cite{MR1473856}).

\section{Proof of Theorem \ref{Theorem 1}}

First we show that $I_\mu$ has a uniformly positive mountain pass level below the threshold for compactness given in Lemma \ref{Lemma 1} for all sufficiently small $\mu > 0$. Let $v_\eps$ be as in the last section.

\begin{lemma} \label{Lemma 2}
There exist $\mu_0, \rho, c_0 > 0$, $R > \rho$, and $\beta < \dfrac{1}{N}\, S^{N/p}$ such that the following hold for all $\mu \in (0,\mu_0)$:
\begin{enumroman}
\item \label{16} $\norm{u} = \rho \implies I_\mu(u) \ge c_0$,
\item \label{17} $I_\mu(t v_\eps) \le 0$ for all $t \ge R$ and $\eps \in (0,1]$,
\item \label{18} denoting by $\Gamma = \set{\gamma \in C([0,1],W^{1,p}_0(\Omega)) : \gamma(0) = 0,\, \gamma(1) = R v_\eps}$ the class of paths joining the origin to $R v_\eps$,
\begin{equation} \label{19}
c_0 \le c_\mu := \inf_{\gamma \in \Gamma}\, \max_{u \in \gamma([0,1])}\, I_\mu(u) \le \beta - \left(1 - \frac{1}{p}\right) \mu \vol{\Omega}
\end{equation}
for all sufficiently small $\eps > 0$,
\item \label{20} $I_\mu$ has a critical point $u_\mu$ at the level $c_\mu$.
\end{enumroman}
\end{lemma}

\begin{proof}
By \eqref{3} and \eqref{5},
\[
I_\mu(u) \ge \frac{1}{p} \left(1 - \frac{\lambda}{\lambda_1}\right)\! \norm{u}^p - \frac{S^{- p^\ast/p}}{p^\ast} \norm{u}^{p^\ast} - \left(1 - \frac{1}{p}\right) \mu \vol{\Omega},
\]
and \ref{16} follows from this for sufficiently small $\rho, c_0, \mu > 0$ since $\lambda < \lambda_1$.

Since $v_\eps \ge 0$,
\[
I_\mu(t v_\eps) = \frac{t^p}{p} \int_\Omega \big(|\nabla v_\eps|^p - \lambda v_\eps^p\big)\, dx - \frac{t^{p^\ast}}{p^\ast} + \mu t \int_\Omega v_\eps\, dx
\]
for $t \ge 0$. By the H\"{o}lder's and Young's inequalities,
\[
\mu t \int_\Omega v_\eps\, dx \le \mu t \vol{\Omega}^{1-1/p} \left(\int_\Omega v_\eps^p\, dx\right)^{1/p} \le C_\lambda\, \mu^{p/(p-1)} + \frac{\lambda t^p}{2p} \int_\Omega v_\eps^p\, dx,
\]
where
\[
C_\lambda = \left(1 - \frac{1}{p}\right)\! \left(\frac{2}{\lambda}\right)^{1/(p-1)} \vol{\Omega},
\]
so
\begin{equation} \label{21}
I_\mu(t v_\eps) \le \frac{t^p}{p} \int_\Omega \left(|\nabla v_\eps|^p - \frac{\lambda}{2}\, v_\eps^p\right) dx - \frac{t^{p^\ast}}{p^\ast} + C_\lambda\, \mu^{p/(p-1)}.
\end{equation}
Then by \eqref{14} and for $\eps, \mu \in (0,1]$,
\[
I_\mu(t v_\eps) \le (S + C)\, \frac{t^p}{p} - \frac{t^{p^\ast}}{p^\ast} + C_\lambda,
\]
from which \ref{17} follows for sufficiently large $R > \rho$.

The first inequality in \eqref{19} is immediate from \ref{16} since $R > \rho$. Maximizing the right-hand side of \eqref{21} over $t \ge 0$ gives
\[
c_\mu \le \frac{1}{N} \left[\int_\Omega \left(|\nabla v_\eps|^p - \frac{\lambda}{2}\, v_\eps^p\right) dx\right]^{N/p} + C_\lambda\, \mu^{p/(p-1)},
\]
and \eqref{14} and \eqref{15} imply that the integral on the right-hand side is strictly less than $S$ for all sufficiently small $\eps > 0$ since $N \ge p^2$ and $\lambda > 0$, so the second inequality in \eqref{19} holds for sufficiently small $\mu > 0$.

Finally, \ref{20} follows from \ref{16}--\ref{18}, Lemma \ref{Lemma 1}, and the mountain pass lemma (see Ambrosetti and Rabinowitz \cite{MR0370183}).
\end{proof}

Next we show that $u_\mu$ is uniformly bounded in $W^{1,p}_0(\Omega)$ and uniformly equi-integrable in $L^{p^\ast}(\Omega)$, and hence also uniformly bounded in $C^{1,\alpha}(\closure{\Omega})$ for some $\alpha \in (0,1)$ by de Figueiredo et al.\! \cite[Proposition 3.7]{MR2530603}, for all sufficiently small $\mu \in (0,\mu_0)$.

\begin{lemma} \label{Lemma 3}
There exists $\mu_\ast \in (0,\mu_0]$ such that the following hold for all $\mu \in (0,\mu_\ast)$:
\begin{enumroman}
\item \label{22} $u_\mu$ is uniformly bounded in $W^{1,p}_0(\Omega)$,
\item \label{23} $\dint_E |u_\mu|^{p^\ast} dx \to 0$ as $\vol{E} \to 0$, uniformly in $\mu$,
\item \label{24} $u_\mu$ is uniformly bounded in $C^{1,\alpha}(\closure{\Omega})$ for some $\alpha \in (0,1)$.
\end{enumroman}
\end{lemma}

\begin{proof}
We have
\begin{multline} \label{25}
I_\mu(u_\mu) = \int_\Omega \bigg(\frac{|\nabla u_\mu|^p}{p} - \frac{\lambda u_{\mu +}^p}{p} - \frac{u_{\mu +}^{p^\ast}}{p^\ast}\bigg)\, dx + \mu\, \bigg[\int_{\set{u_\mu \ge 0}} u_\mu\, dx\\[10pt]
+ \int_{\set{-1 < u_\mu < 0}} \left(u_\mu - \frac{|u_\mu|^{p-1}\, u_\mu}{p}\right) dx - \left(1 - \frac{1}{p}\right) \vol{\set{u_\mu \le -1}}\bigg] = c_\mu
\end{multline}
and
\begin{multline} \label{26}
I_\mu'(u_\mu)\, v = \int_\Omega \left(|\nabla u_\mu|^{p-2}\, \nabla u_\mu \cdot \nabla v - \lambda u_{\mu +}^{p-1}\, v - u_{\mu +}^{p^\ast - 1}\, v\right) dx + \mu\, \bigg[\int_{\set{u_\mu \ge 0}} v\, dx\\[10pt]
+ \int_{\set{-1 < u_\mu < 0}} \left(1 - |u_\mu|^{p-1}\right) v\, dx\bigg] = 0 \quad \forall v \in W^{1,p}_0(\Omega).
\end{multline}
Taking $v = u_\mu$ in \eqref{26}, dividing by $p$, and subtracting from \eqref{25} gives
\begin{equation} \label{27}
\frac{1}{N} \int_\Omega u_{\mu +}^{p^\ast}\, dx \le c_\mu + \left(1 - \frac{1}{p}\right) \mu \vol{\Omega} \le \beta
\end{equation}
by \eqref{19}, and \ref{22} follows from this, \eqref{26} with $v = u_\mu$, and the H\"{o}lder inequality.

If \ref{23} does not hold, then there exist sequences $\mu_j \to 0$ and $\seq{E_j}$ with $\vol{E_j} \to 0$ such that
\begin{equation} \label{28}
\varliminf \int_{E_j} |u_{\mu_j}|^{p^\ast} dx > 0.
\end{equation}
Since $\seq{u_{\mu_j}}$ is bounded by \ref{22}, so is $\seq{u_{\mu_j +}}$, a renamed subsequence of which then converges to some $v \ge 0$ weakly in $W^{1,p}_0(\Omega)$, strongly in $L^q(\Omega)$ for all $q \in [1,p^\ast)$ and a.e.\! in $\Omega$, and
\begin{equation} \label{29}
|\nabla u_{\mu_j +}|^p\, dx \wstar \kappa, \qquad u_{\mu_j +}^{p^\ast}\, dx \wstar \nu
\end{equation}
in the sense of measures, where $\kappa$ and $\nu$ are bounded nonnegative measures on $\closure{\Omega}$. By Lions \cite{MR834360,MR850686}, then there exist an at most countable index set $I$ and points $x_i \in \closure{\Omega},\, i \in I$ such that
\begin{equation} \label{30}
\kappa \ge |\nabla v|^p\, dx + \sum_{i \in I} \kappa_i\, \delta_{x_i}, \qquad \nu = v^{p^\ast}\, dx + \sum_{i \in I} \nu_i\, \delta_{x_i},
\end{equation}
where $\kappa_i, \nu_i > 0$ and $\nu_i^{p/p^\ast} \le \kappa_i/S$. Suppose $I$ is nonempty, say, $i \in I$. An argument similar to that in the proof of Lemma \ref{Lemma 1} shows that $\kappa_i \le \nu_i$, so $\nu_i \ge S^{N/p}$. On the other hand, passing to the limit in \eqref{27} with $\mu = \mu_j$ and using \eqref{29} and \eqref{30} gives $\nu_i \le N \beta < S^{N/p}$, a contradiction. Hence $I = \emptyset$ and
\[
\int_\Omega u_{\mu_j +}^{p^\ast}\, dx \to \int_\Omega v^{p^\ast}\, dx.
\]
As in the proof of Lemma \ref{Lemma 1}, a further subsequence of $\seq{u_{\mu_j}}$ then converges to some $u$ in $W^{1,p}_0(\Omega)$, and hence also in $L^{p^\ast}(\Omega)$, and a.e.\! in $\Omega$. Then
\[
\int_{E_j} |u_{\mu_j}|^{p^\ast} dx \le \int_\Omega \abs{|u_{\mu_j}|^{p^\ast} - |u|^{p^\ast}} dx + \int_{E_j} |u|^{p^\ast} dx \to 0,
\]
contradicting \eqref{28}.

Finally, \ref{24} follows from \ref{22}, \ref{23}, and de Figueiredo et al.\! \cite[Proposition 3.7]{MR2530603}.
\end{proof}

We are now ready to prove Theorem \ref{Theorem 1}.

\begin{proof}[Proof of Theorem \ref{Theorem 1}]
We claim that $u_\mu$ is positive in $\Omega$, and hence a weak solution of problem \eqref{2}, for all sufficiently small $\mu \in (0,\mu_\ast)$. It suffices to show that for every sequence $\mu_j \to 0$, a subsequence of $u_{\mu_j}$ is positive in $\Omega$. By Lemma \ref{Lemma 3} \ref{24}, a renamed subsequence of $u_{\mu_j}$ converges to some $u$ in $C^1_0(\closure{\Omega})$. We have
\begin{multline*}
I_{\mu_j}(u_{\mu_j}) = \int_\Omega \bigg(\frac{|\nabla u_{\mu_j}|^p}{p} - \frac{\lambda u_{\mu_j +}^p}{p} - \frac{u_{\mu_j +}^{p^\ast}}{p^\ast}\bigg)\, dx + \mu_j\, \bigg[\int_{\set{u_{\mu_j} \ge 0}} u_{\mu_j}\, dx\\[10pt]
+ \int_{\set{-1 < u_{\mu_j} < 0}} \left(u_{\mu_j} - \frac{|u_{\mu_j}|^{p-1}\, u_{\mu_j}}{p}\right) dx - \left(1 - \frac{1}{p}\right) \vol{\set{u_{\mu_j} \le -1}}\bigg] = c_{\mu_j} \ge c_0
\end{multline*}
by \eqref{19} and
\begin{multline*}
I_{\mu_j}'(u_{\mu_j})\, v = \int_\Omega \left(|\nabla u_{\mu_j}|^{p-2}\, \nabla u_{\mu_j} \cdot \nabla v - \lambda u_{\mu_j +}^{p-1}\, v - u_{\mu_j +}^{p^\ast - 1}\, v\right) dx + \mu_j\, \bigg[\int_{\set{u_{\mu_j} \ge 0}} v\, dx\\[10pt]
+ \int_{\set{-1 < u_{\mu_j} < 0}} \left(1 - |u_{\mu_j}|^{p-1}\right) v\, dx\bigg] = 0 \quad \forall v \in W^{1,p}_0(\Omega),
\end{multline*}
and passing to the limits gives
\[
\int_\Omega \bigg(\frac{|\nabla u|^p}{p} - \frac{\lambda u_+^p}{p} - \frac{u_+^{p^\ast}}{p^\ast}\bigg)\, dx \ge c_0
\]
and
\[
\int_\Omega \left(|\nabla u|^{p-2}\, \nabla u \cdot \nabla v - \lambda u_+^{p-1}\, v - u_+^{p^\ast - 1}\, v\right) dx = 0 \quad \forall v \in W^{1,p}_0(\Omega),
\]
so $u$ is a nontrivial weak solution of the problem
\[
\left\{\begin{aligned}
- \Delta_p\, u & = \lambda u_+^{p-1} + u_+^{p^\ast - 1} && \text{in } \Omega\\[5pt]
u & = 0 && \text{on } \bdry{\Omega}.
\end{aligned}\right.
\]
Then $u > 0$ in $\Omega$ and its interior normal derivative $\partial u/\partial \nu > 0$ on $\partial \Omega$ by the strong maximum principle and the Hopf lemma for the $p$-Laplacian (see V{\'a}zquez \cite{MR768629}). Since $u_{\mu_j} \to u$ in $C^1_0(\closure{\Omega})$, then $u_{\mu_j} > 0$ in $\Omega$ for all sufficiently large $j$.

It remains to show that $u_\mu$ minimizes $I_\mu$ on $\calN_\mu$ when it is positive. For each $w \in \calN_\mu$, we will construct a path $\gamma_w \in \Gamma$ such that
\[
\max_{u \in \gamma_w([0,1])}\, I_\mu(u) = I_\mu(w).
\]
Since
\[
I_\mu(u_\mu) = c_\mu \le \max_{u \in \gamma_w([0,1])}\, I_\mu(u)
\]
by the definition of $c_\mu$, the desired conclusion will then follow. First we note that the function
\[
g(t) = I_\mu(tw) = \frac{t^p}{p} \int_\Omega \big(|\nabla w|^p - \lambda w^p\big)\, dx - \frac{t^{p^\ast}}{p^\ast} \int_\Omega w^{p^\ast}\, dx + \mu t \int_\Omega w\, dx, \quad t \ge 0
\]
has a unique maximum at $t = 1$. Indeed,
\begin{multline*}
g'(t) = t^{p-1} \int_\Omega \big(|\nabla w|^p - \lambda w^p\big)\, dx - t^{p^\ast - 1} \int_\Omega w^{p^\ast}\, dx + \mu \int_\Omega w\, dx\\[10pt]
= \left(t^{p-1} - t^{p^\ast - 1}\right) \int_\Omega \big(|\nabla w|^p - \lambda w^p\big)\, dx + \left(1 - t^{p^\ast - 1}\right) \mu \int_\Omega w\, dx
\end{multline*}
since $w \in \calN_\mu$, and the last two integrals are positive since $\lambda < \lambda_1$ and $w > 0$, so $g'(t) > 0$ for $0 \le t < 1$, $g'(1) = 0$, and $g'(t) < 0$ for $t > 1$. Hence
\[
\max_{t \ge 0}\, I_\mu(tw) = I_\mu(w) > 0
\]
since $g(0) = 0$. In view of Lemma \ref{Lemma 2} \ref{17}, now it suffices to observe that there exists $\widetilde{R} > \max \set{1,R}$ such that
\[
I_\mu(\widetilde{R} u) = \frac{\widetilde{R}^p}{p} \int_\Omega \big(|\nabla u|^p - \lambda u^p\big)\, dx - \frac{\widetilde{R}^{p^\ast}}{p^\ast} \int_\Omega u^{p^\ast}\, dx + \mu \widetilde{R} \int_\Omega u\, dx \le 0
\]
for all $u$ on the line segment joining $w$ to $v_\eps$ since all norms on a finite dimensional space are equivalent.
\end{proof}

\subsection*{Acknowledgement}
The third author was supported by the National Research Foundation of Korea Grant funded by the Korea Government (MEST) (NRF-2015R1D1A3A01019789).

\def\cdprime{$''$}

\end{document}